\documentclass[submission,copyright,creativecommons]{eptcs}
\providecommand{\event}{Gandalf 2018} 
\usepackage{breakurl}             
\usepackage{underscore}           
\usepackage{amsmath}
\usepackage{amssymb}
\usepackage{graphicx}

\newtheorem{theo}{Theorem}

\newtheorem{defin}[theo]{Definition}
\newtheorem{lemma}[theo]{Lemma}
\newtheorem{propo}[theo]{Proposition}
\newtheorem{coro}[theo]{Corollary}

\newenvironment{proof}{\emph{Proof:}$\\$}{$\\\Box\\$}

\newcommand\M{\mathfrak M}
\newcommand\Var{\texttt{Var}}
\newcommand\dom{\texttt{Dom}}
\newcommand\tuple{\mathbf}

\newcommand {\parts}{\mathcal P}

\newcommand \nonempty{\texttt{NE}}
\newcommand{\ESO}{\text{ESO}}
\newcommand{\FO}{\text{FO}}

\newcommand{\ar}{\texttt{ar}}
\newcommand{\fv}{\texttt{FV}}
\newcommand{\var}{\texttt{var}}
\newcommand{\nc}{\texttt{NC}}
\newcommand{\All}{\texttt{All}}

\title{Safe Dependency Atoms and Possibility Operators in Team Semantics}
\author{Pietro Galliani
\institute{Free University of Bozen-Bolzano\\ Bolzano, Italy}
\email{Pietro.Galliani@unibz.it}
}
\def\titlerunning{Safe Dependency Atoms}
\def\authorrunning{P. Galliani}
\begin{document}
\maketitle

\begin{abstract}
	I consider the question of which dependencies are \emph{safe} for a Team Semantics-based logic $FO(\mathcal D)$, in the sense that they do not increase its expressive power over sentences when added to it. I show that some dependencies, like totality, non-constancy and non-emptiness, are safe for all logics $FO(\mathcal D)$, and that other dependencies, like constancy, are not safe for $FO(\mathcal D)$ for some choices of $\mathcal D$ despite being strongly first order (that is, safe for $FO(\emptyset)$). I furthermore show that the \emph{possibility operator} $\diamond \phi$, which holds in a team if and only if $\phi$ holds in some nonempty subteam, can be added to any logic $\FO(\mathcal D)$ without increasing its expressive power over sentences. 
\end{abstract}

\section{Introduction}
Team Semantics \cite{hodges97} generalizes Tarskian Semantics for First Order Logic by allowing formulas to be satisfied or not satisfied with respect to sets of assignments (called \emph{teams}), rather than with respect to single assignments. First Order Logic with Team Semantics is easily shown to be equivalent to First Order Logic with Tarskian Semantics, in the sense that a first order formula is satisfied by a set of assignments in Team Semantics if and only if it is satisfied by all assignments in the set with respect to Tarskian Semantics. 

The richer nature of the satisfaction relation of Team Semantics, however, makes it possible to extend First Order Logic in novel ways, such as by introducing new operators or quantifiers \cite{barbero2017some,engstrom12,galliani2013epistemic,vaananen07b} or new types of atomic formulas which specify dependencies between different assignments contained in a team. Examples of important logics obtained in the latter way are \emph{Dependence Logic} \cite{vaananen07}, \emph{Inclusion Logic} \cite{galliani12}, and \emph{Independence Logic} \cite{gradel13}. Despite the semantics of the atoms which these logics add to the language of First Order Logic being first order (when understood as conditions over the relations corresponding to teams), these logics are strictly more expressive than First Order Logic. This, in brief, is due to the second order existential quantifications implicit in the Team Semantics rules for disjunction and existential quantification. Thus, exploring the properties of fragments of such logics (as done for instance in \cite{durand2012hierarchies,durand2016expressivity,galliani13b,hannula2015hierarchies,ronnholm2015capturing}) provides an interesting avenue to the study of the properties and relations between fragments of Second Order Logic. 

This work is a contribution towards the more systematic study of the properties of first order definable dependency atoms and of the logics they generate. Building on the work of \cite{galliani13e,galliani2016strongly}, which dealt with the case of dependencies which are \emph{strongly first order} in that they do not increase the expressive power of First Order Logic if added to it, we will find some preliminary answers to the following

\paragraph{Question:} Let $\mathcal D = \{\mathbf D_1, \mathbf D_2, \ldots\}$ be a set of first order definable dependencies. Can we characterize the sets of dependencies $\mathcal E = \{\mathbf E_1, \mathbf E_2, \ldots\}$ which are \textbf{safe} for $\mathcal D$, in the sense that every sentence of $FO(\mathcal D, \mathcal E)$ is equivalent to some sentence of $FO(\mathcal D)$?\\

To the author's knowledge, this notion of safety  -- which is the natural generalization of the notion of strongly first order dependency of \cite{galliani13e,galliani2016strongly} -- has not been considered so far in the literature; and, as we will see, known results  and currently open problems regarding the expressive power of logics with Team Semantics can be reframed in terms of it, and information concerning the safety of dependencies (or operators, if we generalize the notion of dependency to operators in the obvious way) can be highly useful to prove relationships between logics with Team Semantics. However, as we will also see, safety is a delicate notion: in particular, dependencies which are strongly first order (that is, safe for the empty set of dependencies) are not necessarily safe for all sets of dependencies.

These results will show that this notion of safety is a subtle one, deserving of further investigation. Additionally, by means of these answers we will see that the possibility operator $\diamond \phi$, which holds in a team if $\phi$ holds in some nonempty subteam of it, can be added to any logic $\FO(\mathcal D)$ without increasing its expressive power. 
\section{Preliminaries}
\subsection{Team Semantics}
In this section we will briefly recall the notation used in this work, the definition of Team Semantics, and some basic results that will be used in the rest of this work. Through all of this work, we will always assume that all our (first order) models $\M$ have at least two elements in their domain $M$ and that we have countable sets of variable symbols $\{x_i, y_i, z_i, w_i, \ldots : i \in \mathbb N\}$ and of relation symbols $R, S, \ldots$ of all arities. We will write $\tuple x$, $\tuple y$, $\tuple v$ and so on to describe tuples of variable symbols; and likewise, we will write $\tuple m$, $\tuple a$, $\tuple b$ and so forth to describe tuples of elements of a model. For any tuple $\tuple a$ of elements, $|\tuple a|$ will represent the \emph{length} of $\tuple a$; and likewise, $|\tuple v|$ represents the length of the tuple of variables $\tuple v$. Given any set $A$, we will furthermore write $\parts(A)$ for the powerset $\{B : B \subseteq A\}$ of $A$.

Variable assignments and substitutions are defined in the usual way: 
\begin{defin}[Variable Assignments, Substitution, Restriction, Composition with Functions]
	Let $\M$ be a first order model with domain $M$ and let $V$ be a set of variables. Then an \emph{assignment} over $\M$ with domain $\dom(s) = V$ is a function $s : V \rightarrow M$. We will write $\epsilon$ for the unique assignment with domain $\emptyset$. For any variable $v$ (which may or may not be in $V$ already) and any element $m \in M$, we write $s[m/v]$ for the variable assignment with domain $V \cup \{v\}$ such that  
	\[
		s[m/v](x) = \left\{\begin{array}{l l}
			m & \text{ if } x = v; \\
			s(x) & \text{ otherwise}
		\end{array}
			\right.
	\]
	for all variable symbols $x \in V \cup \{v\}$.

	For every assignment $s$, every tuple $\tuple m = m_1 \ldots m_n$ of elements and every tuple $\tuple v = v_1 \ldots v_n$ of variables with $|\tuple v| = |\tuple m|$, we will write $s[\tuple m/\tuple v]$ as an abbreviation for $s[m_1/v_1][m_2 /v_2] \ldots [m_n/v_n]$.

	For all sets of variables $V \subseteq \dom(s)$, we furthermore write $s_{|V}$ for the \emph{restriction} of the assignment $s$ to the variables of $V$, that is, for the unique assignment $s'$ with domain $V$ such that $s'(v) = s(v)$ for all $v \in V$.

	For any function $\mathfrak f: M \rightarrow M$ and any assignment $s$ over $\M$, we will write $\mathfrak f(s)$ for the unique assignment with the same domain of $s$ such that $\mathfrak f(s)(v) = \mathfrak f(s(v))$ for all $v \in \dom(s)$. 
\end{defin}

Given an expression $\phi$, we will write $\fv(\phi)$ for the set of all variables occurring free (that is, not in the scope of a quantifier for them) in $\phi$; and given a tuple $\tuple t$ of terms of our language, we will write $\var(\tuple t)$ for the set of all variables occurring in $\tuple t$.

Let us now recall the definition of Team Semantics for First Order Logic: 
\begin{defin}[Team]
	Let $V$ be a finite set of variables, and let $\M$ be a first order model with domain $M$. Then a team $X$ over $\M$ with domain $\dom(X) = V$ is a set of assignments $s: V \rightarrow M$. 
\end{defin}
\begin{defin}[From Teams to Relations]
	Let $X$ be a team over a first order model $\M$, and let $\tuple v = v_1 \ldots v_k$ be a tuple (possibly with repetitions) of variables $v_i \in \dom(X)$. Then we write $X(\tuple v)$ for the $k$-ary relation given by
	\[
		X(\tuple v) = \{s(v_1v_2\ldots v_n)  : s \in X\}
	\]
	where $s(v_1 v_2 \ldots v_n)$ is a shorthand for the $n$-tuple $(s(v_1), s(v_2), \ldots, s(v_n))$. 
\end{defin}
\begin{defin}[Team Supplementation]
	Let $X$ be a team over some first order model $\M$, let $k \in \mathbb N$, and let $\tuple v \in \Var^k$ be a tuple of $k$ distinct variables (which may or may not occur already in $\dom(X)$).Then, for all functions $H: X \rightarrow \parts(\dom(X)^k) \backslash \{\emptyset\}$, we define $X[H/\tuple v]$ as the team with domain $\dom(X) \cup \tuple v$ given by 
	\[
		X[H/\tuple v] = \{s[\tuple m/\tuple v]: s \in X, \tuple m \in H(s)\}. 
	\]
\end{defin}
In other words, a supplementation function $H$ for the team $X$ selects, for each assignment $s \in X$, a nonempty set $H(X)$ of possible values for the variables $\tuple v$, and $X[H/\tuple v]$ is obtained from $X$ by assigning these possible values to the variables $\tuple v$.

The \emph{duplication} operator, which will be now described, corresponds then to the special case of supplementation for which $H(s) = M^k$ for all $s \in X$: 
\begin{defin}[Team Duplication]
	Let $X$ be a team over some first order model $\mathfrak M$, let $k \in \mathbb N$, and let $\tuple v \in \Var^k$ be once more a tuple of $k$ distinct variables. Then the duplication $X[M/\tuple v]$ of $X$ along $\tuple v$ is the team 
	\[	
		X[M/\tuple v] = \{s[\tuple m/\tuple v] : s \in X, \tuple m \in M^k\}.
	\]
\end{defin}

Team Semantics was originally developed by Hodges in \cite{hodges97} in order to provide a compositional semantics equivalent to the imperfect-information, game-theoretic semantics of \emph{Independence-Friendly Logic} \cite{hintikka96,hintikkasandu89,hintikkasandu97}; but for our purposes it will be useful to first present it for First Order Logic proper. For simplicity, we will assume that all expressions are in \emph{Negation Normal Form:} 
\begin{defin}[Team Semantics for First Order Logic]
	Let $\M$ be a first order model with domain $M$, let $\phi(\tuple x)$ be a first order formula in negation normal form with free variables contained in $\tuple x$, and let $X$ be a team over $\M$ with domain $\dom(X) \supseteq \tuple x$.\footnote{We use this slight abuse of notation to mean that every variable $x_i$ occurring in the tuple $\tuple x = x_1 \ldots x_n$ belongs to $\dom(X)$.} Then we say that the team $X$ satisfies $\phi(\tuple x)$ in $\M$, and we write $\M \models_X \phi$, if this can be derived via the following rules: 
	\begin{description}
		\item[TS-lit:] For all first order literals $\alpha$, $\M \models_X \alpha$ if and only if, for all assignments $s \in X$, $\M \models_s \alpha$ according to the usual Tarskian Semantics; 
		\item[TS-$\vee$:] For all formulas $\psi_1$ and $\psi_2$, $\M \models_X \psi_1 \vee \psi_2$ if and only if there exist teams $Y$ and $Z$ such that $X = Y \cup Z$, $\M \models_Y \psi_1$ and $\M \models_Z \psi_2$; 
		\item[TS-$\wedge$:] For all formulas $\psi_1$ and $\psi_2$, $\M \models_X \psi_1 \wedge \psi_2$ if and only if $\M \models_X \psi_1$ and $\M \models_X \psi_2$; 
		\item[TS-$\exists$:] For all variables $v$ and formulas $\psi$, $\M \models_X \exists v \psi$ if and only if there exists some $H: X \rightarrow \parts(M) \backslash \{\emptyset\}$ such that $\M \models_{X[H/v]} \psi$; 
		\item[TS-$\forall$:] For all variables $v$ and formulas $\psi$, $\M \models_X \forall v \psi$ if and only if $\M \models_{X[M/v]} \psi$.
	\end{description}

	If $\phi$ is a sentence (i.e. has no free variables), we say that $\phi$ is \emph{true} in $\M$ according to Team Semantics, and we write $\M \models \phi$, if and only if $\M \models_{\{\epsilon\}} \phi$, where $\{\epsilon\}$ is the team containing only the empty assignment. 
\end{defin}
It is worth remarking that the above semantics for the language of first order logic involves \emph{second order} existential quantifications in the rules \textbf{TS-$\vee$} and \textbf{TS-$\exists$}. This is a crucial fact for understanding the expressive power of logics based on Team Semantics, and it is furthermore the reason why Team Semantics constitutes a viable tool for describing and studying fragments of existential second order logic. Nonetheless, as the following well known result shows, there exists a very strict relationship between the satisfaction conditions of first order formulas in Team Semantics and in the usual Tarskian Semantics:
\begin{propo}
	Let $\M$ be a first order model, let $\phi$ be a first order formula over the signature of $\M$, and let $X$ be a team over $\M$ with domain containing all the free variables of $\phi$. Then $\M \models_X \phi$ if and only if for all $s \in X$, $\M \models_s \phi$ according to the usual Tarskian Semantics. 
	\label{propo:fo}
\end{propo}
It is a straightforward consequence of Proposition \ref{propo:fo} that truth in Tarskian Semantics and in Team Semantics coincide: 
\begin{coro}
	Let $\M$ be a first order model and let $\phi$ be a first order sentence. Then $\M \models \phi$ according to Team Semantics if and only if $\M \models \phi$ according to the usual Tarskian Semantics. 
\end{coro}
\subsection{The $[R : \tuple t]$ operator, dependencies, and a normal form}
As we saw in the previous section, there is a very strict connection between Tarskian Semantics and Team Semantics for First Order Logic: not only these two semantics agree with respect to the truth of sentences, but the satisfaction conditions of a first order formula $\phi$ with respect to Team Semantics can be obtained in a very straightforward way from the satisfaction conditions of the same formula with respect to Tarskian Semantics. 

There is, however, an important asymmetry in First Order Logic between Tarskian Semantics and Team Semantics. Every first order definable property of tuples of elements corresponds trivially to the satisfaction condition (in Tarskian Semantics) of some first order formula. However, not all first order definable properties of \emph{teams} (interpreted as relations) correspond to the satisfaction conditions (in Team Semantics) of first order formulas, as the following easy consequence of Proposition \ref{propo:fo} shows:
\begin{coro}
	There is no first order formula $\phi(v)$, with $v$ as its only free variable, such that for all first order models $\M$ and teams $X$ with $v \in \dom(X)$ it holds that $\M \models_X \phi(v)$ if and only if $|X(v)| = |\{s(v): s \in X\}| \geq 2$ (that is, if and only if the variable $v$ takes at least two distinct values in $X$). 
	\label{coro:nc}
\end{coro}

Thus, the property of unary relations describable as ``containing at least two elements'', which is easily seen to be first order definable via the sentence $\Phi(U) = \exists p q (U p \wedge U q \wedge p \not = q)$, does not correspond to the satisfaction conditions (according to Team Semantics) of any first order formula.

A straightforward way to ensure that all first order definable properties of relations correspond to the satisfaction conditions of formulas would be to add the following rule to our semantics: 
\begin{description}
	\item[TS-{[:]}:] For all signatures $\Sigma$, all models $\M$ having signature $\Sigma$, all $k \in \mathbb N$, all $k$-ary relation symbols $R$ (which may or may not occur already in $\Sigma$), all tuples $\tuple t = t_1 \ldots t_k$ of terms, and all first order formulas $\phi$ in the signature $\Sigma \cup \{R\}$,
	\[
		\M \models_X [R : \tuple t]\phi \text{ if and only if } \M[X(\tuple t)/R] \models_X \phi 
	\]
	where $\M[X(\tuple t)/R]$ is the expansion of $\M$ to the signature $\Sigma \cup \{R\}$ such that its interpretation $R^{\M[X(\tuple t)/R]}$  of $R$ is simply $X(\tuple t)$. 
\end{description}
Much of the study of Team Semantics so far has focused on the classification of logics obtained by adding expressions of the form $[R:t]\phi$ to First Order Logic, $\phi$ belongs to some class of first order sentences over the signature $\{R\}$.\footnote{Exceptions to this are given for instance by the study of logics which add to Team Semantics \emph{generalised quantifiers} \cite{barbero2017some,engstrom12}, or a \emph{contradictory negation} \cite{vaananen07b}.} 

\begin{defin}[(First Order) Dependencies]
	Let $k \in \mathbb N$. A $k$-ary first order dependency $\mathbf D$ is a first order sentence $\mathbf D(R)$ over the signature $\{R\}$, where $R$ is a $k$-ary relation symbol.\footnote{This is a special case of the more general -- and not necessarily first order -- notion of dependency used in \cite{galliani13e}, which comes from \cite{kuusisto13}.}
\end{defin}
\begin{defin}[$\FO(\mathcal D)$]
	Let $\mathcal D  = \{\mathbf D_1 \ldots \mathbf D_n\}$ be a family of first order dependencies. Then $\FO(\mathcal D)$ is obtained by adding to First Order Logic (with Team Semantics) all \emph{dependency atoms} of the form\\$[R: \tuple t] \mathbf D_i(R)$ for all $i = 1 \ldots n$, where $\tuple t$ is a tuple of terms the same arity $\ar(\mathbf D_i)$ of $\mathbf D_i$, $R$ is a relational symbol of the same arity, and we write $\mathbf D_i \tuple t$ as a shorthand for $[R : \tuple t] \mathbf D_i(R)$. 
\end{defin}

We conclude this section with some simple results that are easily shown to hold for the full $\FO([:])$ and for all its fragments (including all $\FO(\mathcal D)$), and with a \emph{normal form} for all sentences in $\FO(\mathcal D)$ for any set $\mathcal D$ of dependencies:

\begin{defin}[Properties of Formulas and Dependencies]
	Let $\phi(\tuple v)$ be any formula of $\FO[:]$. Then we say that $\phi$ 
	\begin{itemize}
		\item is \textbf{Downwards Closed} if $\M \models_X \phi, Y \subseteq X \Rightarrow \M \models_Y \phi$ for all suitable models $\M$ and teams $X, Y$; 
		\item is \textbf{Upwards Closed} if $\M \models_X \phi, Y \supseteq X \Rightarrow \M \models_Y \phi$ for all suitable models $\M$ and teams $X, Y$; 
		\item is \textbf{Union Closed} if $\M \models_{X_i} \phi \forall i \in I \Rightarrow \M \models_{\cup_i X_i} \phi$ for all suitable models $\M$ and families of teams (all with the same domain) $(X_i)_{i \in I}$; 
		\item has \textbf{the Empty Team Property} if $\M \models_{\emptyset} \phi$.
	\end{itemize}
	We say that a dependency $\mathbf D$ (that is, a first order sentence $D(R)$ over the signature $\{R\}$) has any such property if all the formulas $\textbf D \tuple t$ (that is, $[R:\tuple t]\mathbf D(R)$) have it.
	\label{def:props}
\end{defin}
Three of these four properties are preserved by the connectives of our language, as it can be proved by straightforward induction: 
\begin{propo}
	Let $\mathcal D = \{\mathbf D_1, \mathbf D_2, \ldots\}$ be a family of dependencies which are all Downwards Closed [are all Union Closed, have all the Empty Team Property]. Then every formula of $\FO(\mathcal D)$ is Downwards Closed [is Union Closed, has the Empty Team Property].
\label{propo:pres}
\end{propo}
The property of union closure, on the other hand, is clearly not preserved in the same way as it is violated already by first order literals. However, this property is nonetheless useful for the classification of the expressive power of logics with Team Semantics. 

\begin{defin}[Team Restriction] 
	Let $X$ be a team over a model $\M$, and let $V \subseteq \dom(X)$. Then $X_{|V}$ is the \emph{restriction} of $X$ to the domain $V$, that is, the team 
		$X_{|V} = \{s_{|V} : s \in X\}$. 
\end{defin}
\begin{propo}[Locality]
	Let $\M$ be any first order model, let $\phi \in \FO([:])$ be a formula over the signature of $\M$, and let $X$ be a team over $\M$ such that the set $\fv(\phi)$ of the free variables of $\phi$ is contained in $\dom(X)$. Then $\M \models_X \phi$ if and only if $\M \models_{X_{|\fv(\phi)}} \phi$. 
	\label{propo:local}
\end{propo}
The following result is the generalization to $\FO([:])$ of Proposition 19 of \cite{galliani13b}, and the proof is entirely analogous: 
\begin{propo}
	The following equivalences hold for all $\psi_1, \psi_2 \in \FO([:])$ and all variables $v$ occurring free in $\psi_1$ but not in $\psi_2$ and for all two variables $p$ and $q$, different from each other and from $v$, which occur in neither $\psi_1$ nor $\psi_2$;
	\begin{enumerate}
		\item $(\exists v \psi_1) \vee \psi_2 \equiv \exists v (\psi_1 \vee \psi_2)$; 
		\item $(\exists v \psi_1) \wedge \psi_2 \equiv \exists v (\psi_1 \wedge \psi_2)$; 
		\item $(\forall v \psi_1) \vee \psi_2 \equiv \exists p q \forall v ( (p = q \wedge \psi_1) \vee (p \not = q \wedge \psi_2))$; 
		\item $(\forall v \psi_1) \wedge \psi_2 \equiv \forall v (\psi_1 \wedge \psi_2)$
	\end{enumerate}
	\label{propo:qmove}	
\end{propo}

It follows from the above equivalences that all logics $\FO(\mathcal D)$, for all choices of $\mathcal D$, admit the following Prenex Normal Form, which is analogous of the one proved in Theorem 15 of \cite{galliani13b}: 
\begin{theo}
	Let $\mathcal D$ be any family of dependencies, and let $\phi$ be a formula of $\FO(\mathcal D)$. Then $\phi$ is logically equivalent to some formula $\phi' \in \FO(\mathcal D)$ of the form $Q_1 v_1 \ldots Q_n v_n \psi$, where each $Q_i$ is $\exists$ or $\forall$ and $\psi$ is quantifier-free. Furthermore, $\psi$ contains the same number of dependency atoms that $\phi$ does, and the number of universal quantifiers among $Q_1 \ldots Q_n$ is the same as the number of universal quantifiers in $\phi$ (although there may be more existential quantifiers in $Q_1 \ldots Q_n$ than in $\psi$). 
	\label{thm:pnf}
\end{theo}
Theorem \ref{thm:normalform} at the end of this section will show how this normal form may be further refined.

\begin{defin}[Team Conditioning]
	Let $X$ be a team over a model $\M$ and let $\theta(\tuple v)$ be a first order formula with free variables in $\dom(X)$. Then $X\upharpoonright \theta$ is the subteam of $X$ containing only the assignments which satisfy $\theta$ (in the Tarskian Semantics sense), that is, 
	\[
		X\upharpoonright \theta = \{s \in X : \M \models_s \theta\}
	\]
\end{defin}
\begin{defin}[$\theta \hookrightarrow \phi$]
	Let $\theta$ be a first order formula with free variables in $\tuple x$ and let $\phi$ be a $\FO([:])$ formula. Then we define $\theta \hookrightarrow \phi$ as $(\lnot \theta) \vee (\theta \wedge \phi)$, where $\lnot \theta$ is the first order negation normal form expression equivalent to the negation of $\theta$. 
\end{defin}
In general, in Team Semantics $\theta \hookrightarrow \phi$ is not logically equivalent to the typical interpretation $\lnot \theta \vee \phi$ of the implication $\theta \rightarrow \phi$.\footnote{It is so if $\phi$ is downwards closed.} In \cite{galliani13e,galliani2016strongly} the same operator was written as $\phi \upharpoonright \theta$; here, however, we prefer to use the $\hookrightarrow$ notation as in the first occurrence of an operator of this type in the literature\footnote{The $\hookrightarrow$ operator of \cite{kontinennu09} had a more general semantics in order to deal with non first-order in the antecedent -- in short, according to \cite{kontinennu09} $\M \models_X \theta \hookrightarrow \theta$ if and only if $\M \models_Y \theta$ for all \emph{maximal} $Y \subseteq X$ which satisfy $\theta$. If $\theta$ is first order, it follows easily from Proposition \ref{propo:fo} that this is equivalent to definition given above.} \cite{kontinennu09} and as in recent literature in the area of Team Semantics (e.g. \cite{luck2018complexity}), in order to emphasize the ``implication-like'' qualities of this connective. 
\begin{propo}
	For all first order formulas $\theta$ and all formulas $\phi \in \FO([:])$, $\M \models_X \theta \hookrightarrow \phi$ if and only if $\M \models_{X\upharpoonright \theta} \phi$.
	\label{propo:restrict}
\end{propo}

As long as we are working with models with at least two elements it is possible to use the $\hookrightarrow$ operator to get rid of the second order quantification implicit in the Team Semantics rule for disjunctions, at the cost of adding further existential quantifiers: 
\begin{lemma}
	Let $\psi_1$ and $\psi_2$ be two formulas of $\FO([:])$, and let $q_1$, $q_2$ be two variables not occurring in either $\psi_1$ or $\psi_2$. Then $\psi_1 \vee \psi_2$ is logically equivalent to $\exists q_1 q_2 ((q_1 = q_2 \hookrightarrow \psi_1) \wedge (q_1 \not = q_2 \hookrightarrow \psi_2))$ over models with at least two elements. 
	\label{lemma:disj2harp}
\end{lemma}

Furthermore, the $\hookrightarrow$ operator commutes with the other operators: 
\begin{lemma}
	For all formulas $\theta, \theta_1, \theta_2 \in \FO$ and $\psi, \psi_1, \psi_2 \in \FO([:])$, 
\begin{itemize}
	\item $\theta_1 \hookrightarrow (\theta_2 \hookrightarrow \psi) \equiv (\theta_1 \wedge \theta_2) \hookrightarrow \psi$. 
\item $\theta \hookrightarrow (\psi_1 \wedge \psi_2) \equiv (\theta \hookrightarrow \psi_1) \wedge (\theta \hookrightarrow \psi_2)$; 
\item If the variable $y$ does not occur in $\theta$ then $\theta \hookrightarrow (\exists y \psi)\equiv \exists y (\theta \hookrightarrow \psi)$; 
\item If the variable $y$ does not occur in $\theta$ then $\theta \hookrightarrow (\forall y \psi) \equiv \forall y (\theta \hookrightarrow \psi)$.
\end{itemize}
	\label{lemma:harpinside}
\end{lemma}

Using the above results it is possible to prove the existence of the following normal form:
\begin{theo}[Normal Form for $\FO(\mathcal D)$]
	Let $\mathcal D = \{\mathbf D_1, \mathbf D_2, \ldots\}$ be any set of dependencies and let  $\phi$ be a sentence of $\FO(\mathcal D)$. Then $\phi$ is logically equivalent to some sentence $\phi'$ of the form 
\[
	\forall \tuple  x_1 \exists \tuple y_1 \ldots \forall \tuple x_n \exists \tuple y_n (\bigwedge_k (\theta_k(\tuple y_n) \hookrightarrow \mathbf D_{i_k} \tuple t_k) \wedge \psi(\tuple x, \tuple y)).
\]
	where the $\theta_k$ and $\psi$ are quantifier-free and contain no dependency atoms, and where furthermore each possible instance $\mathbf D_i \tuple t$ of every dependence atom $\mathbf D_i \in \mathcal D$ appears the same number of times in $\phi$ and in $\phi'$ and there are as many universal quantifiers in $\phi'$ as in $\phi$
	\label{thm:normalform}
\end{theo}
\begin{proof}
	First, let us rename variables so that no variable is bound in two different places in $\phi$ and no variable occurs both bound and free in $\phi$.\footnote{We do not discuss in detail here the effect of renaming variables in logics with Team Semantics, and remark only that there is no substantial difference between such logics and first order logic in this respect.} Then let us bring $\phi$ in prenex normal form $Q_1 v_1 \ldots Q_n v_n \psi$ for $\psi$ quantifier free, as per Theorem \ref{thm:pnf}. 

	Then let us get rid of disjunctions by Lemma \ref{lemma:disj2harp}, replacing each subformula $\psi_1 \vee \psi_2$ with $\exists q_1 q_2 (q_1 = q_2 \hookrightarrow \psi_1) \wedge (q_1 \not = q_2 \hookrightarrow \psi_2)$ for two new variables $q_1$ and $q_2$ (different for each disjunction). Then let us bring the newly introduced existential quantifiers outside of subexpressions too, using the transformations of Proposition \ref{propo:qmove} and Lemma \ref{lemma:harpinside} as required. Finally, again using the transformations of Lemma \ref{lemma:harpinside}, let us bring conjunctions outside the consequents of $\hookrightarrow$ operators and merge multiple occurrences of $\hookrightarrow$ of the form $\theta_1 \hookrightarrow(\theta_2 \hookrightarrow \psi)$ as $(\theta_1 \wedge \theta_2) \hookrightarrow \psi$. 

The final result will be an expression of the form
$\texttt Q_1 v_1 \ldots \texttt Q_n v_n \exists \tuple y \bigwedge_j (\theta_j(\tuple y) \hookrightarrow \alpha_j(\tuple v, \tuple y))$, 
	where all $\alpha_j$ are either occurrences $\mathbf D \tuple t$ of dependency atoms $\mathbf D \in \mathcal D$ or first order literals $\alpha$ and where the $\theta_j$ are quantifier-free conjunctions of first order literals with variables in $\tuple y$. This is easily seen to be the same as the required form, where we combined all $\theta_j \hookrightarrow \alpha_j$ for first order $\alpha_j$ into $\psi$. It is clear furthermore that no additional universal quantifiers or dependency atoms are introduced by this transformation.
\end{proof}

\subsection{Strongly First Order Dependencies}
\label{sec:stronglyfo}
Because of the higher order quantification hidden in the Team Semantics rules for disjunction and existential quantification, even comparatively simple first order dependencies such \emph{inclusion atoms} \cite{galliani12} $\tuple x \subseteq \tuple y := [R: \tuple x \tuple y] \forall \tuple u \tuple v (R \tuple u \tuple v \rightarrow \exists \tuple w R \tuple w \tuple u)$ or \emph{functional dependence atoms} \cite{vaananen07} $=\!\!(\tuple x ; y) := [R : \tuple x y] \forall \tuple u v_1 v_2 (R\tuple u v_1 \wedge R \tuple u v_2 \rightarrow v_1 = v_2)$ bring the expressive power of the logic well beyond that of First Order Logic. 

A dependency, or set of dependencies, is said to be \emph{strongly first order} if this is not the case: 
\begin{defin}[Strongly First Order Dependencies]
	Let $\mathcal D = \{\mathbf D_1, \mathbf D_2, \ldots\}$ be a set of dependencies. We say that $\mathcal D$ is \textbf{strongly first order} if and only if every sentence of $\FO(\mathcal D)$ is logically equivalent to some sentence of First Order Logic $\FO$. 
\end{defin}

It is important to emphasize here that the above definition asks merely that every \emph{sentence} of $\FO(\mathcal D)$ is equivalent to some sentence of $\FO$. As we saw in Corollary \ref{coro:nc}, not all first order properties of teams correspond to the satisfaction conditions of first order formulas in Team Semantics; but nonetheless, some of those properties may be added as dependencies to First Order Logic without increasing the expressive power of its sentences. We can ask then the following\\

\noindent \textbf{Question: } Are there non-trivial choices of $\mathcal D$ which are strongly first order?\\ 

This is a question of some importance not only because of its relevance to the classification of extensions of First Order Logic via Team Semantics but also because knowing which families of dependencies do not make the resulting logics computationally untreatable is essential for studying applications of Team Semantics in e.g. Database Theory (see for example \cite{kontinen13}). 

A positive answer to the above question was found in \cite{galliani13e}, in which the following result was found: 
\begin{theo}
	Let $\mathcal D^{\uparrow}$ be the family of all \emph{upwards closed} dependencies\footnote{That is, as per Definition \ref{def:props}, all $\mathbf D(R) \in \mathcal D$ must be such that $(M, R) \models \mathbf D(R), R \subseteq S \Rightarrow (M, S) \models \mathbf D(S)$; or equivalently, in terms of Team Semantics, all $\mathbf D \in \mathcal D$ are such that $\M \models_X \mathbf D \tuple t, X \subseteq Y \Rightarrow \M \models_Y \mathbf D \tuple t$.} and let $=\!\!(\cdot)$ be the family of all \emph{constancy dependencies}
	\[
		=\!\!(\cdot) := \forall \tuple x \tuple y (R \tuple x \wedge R \tuple y \rightarrow \tuple x = \tuple y)
	\]
	of all arities.\footnote{It is straightforward, however, to see that constancy dependencies of arity one suffice to define the others: for instance, $=\!\!(xy) \equiv =\!\!(x) \wedge =\!\!(y)$.} Then $\mathcal D^{\uparrow} \cup =\!\!(\cdot)$ is strongly first order. 
	\label{thm:uc}
\end{theo}
In \cite{galliani2016strongly} it was furthermore shown that all \emph{unary} first-order dependencies are definable in $\FO(\mathcal D^{\uparrow}, =\!\!(\cdot))$,\footnote{In this work we will commit a minor notational abuse here and write $\FO(\mathcal D^{\uparrow}, =\!\!(\cdot))$ instead of $\FO(\mathcal D^{\uparrow} \cup =\!\!(\cdot))$ and so forth.} and hence do not increase the expressive power of First Order Logic if added to it. 

It is still unknown, however, whether the above result is a characterization of \emph{all} strongly first order families of dependencies. In other words, the following problem is still open:\\

\noindent \textbf{Open Conjecture:} Let $\mathcal D$ be a strongly first order family of dependencies. Then every $\mathbf D \in \mathcal D$ is definable in $\FO(=\!\!(\cdot), \mathcal D^\uparrow)$.

\section{Safe Dependencies}

By definition, a class $\mathcal D$ of dependencies is strongly first order if and only if $\FO(\mathcal D)$ is no more expressive than $\FO$ over sentences. In many cases, this is perhaps too restrictive a notion: indeed, it may be that instead we have already a family $\mathcal D$ of dependencies whose expressive power is suitable for our needs (for instance, as in the case of inclusion dependencies, that captures the PTIME complexity class over finite ordered structures) and we may be interested in characterizing the families $\mathcal E$ that do not further increase it if added to the language. This justifies the following, more general notion: 

\begin{defin}
Let $\mathcal D = \{\mathbf D_1 \ldots \mathbf D_n\}$ be a set of dependencies. Another set of dependencies $\mathcal E$ is \textbf{safe} for $\mathcal D$ if any sentence of $\FO(\mathcal E, \mathcal D)$ is equivalent to some sentence of $\FO(\mathcal D)$. 
%
\end{defin}

It is obvious that strongly first orderness is a special case of safety: 
\begin{propo}
	A family $\mathcal D$ of dependencies is strongly first order if and only if it is safe for the empty set of dependencies $\emptyset$.
\end{propo}

Furthermore, it is trivial to see that definable dependencies are always safe: 
\begin{propo}[Definable Dependencies are safe]
	Let $\mathcal D$ and $\mathcal E$ be two families of dependencies such that for all $\mathbf E \in \mathcal E$ there exists some formula $\psi_{\mathbf E}(\tuple v) \in \FO(\mathcal D)$ such that 
		$\M \models_X \mathbf E \tuple v \Leftrightarrow \M \models_X \psi_{\mathbf E}(\tuple v)$
	for all models $\M$, tuples $\tuple v$ of distinct variables of length equal to the arity of $\mathbf E$, and teams $X$ over $\M$ with domain $\tuple v$.  Then $\mathcal E$ is safe for $\mathcal D$.
\end{propo}

Are all dependencies (or families of dependencies) which are safe for some $\mathcal D$ definable in it? In general, this cannot be true: as we saw in Corollary \ref{coro:nc}, non-constancy dependencies 
\[
	\nc(\tuple v) := [R : \tuple v] \exists \tuple x \tuple y (R \tuple x \wedge R \tuple y \wedge \tuple x \not = \tuple y)
\]
are not definable in $\FO = \FO(\emptyset)$, but since they are upwards closed we know by Theorem \ref{thm:uc} that they are strongly first order (and, therefore, safe for $\emptyset$). Or, to mention another example, \emph{all} families of dependence atoms are safe for the functional dependence atoms of Dependence Logic: indeed, Dependence Logic is equivalent to full Existential Second Order Logic $\Sigma_1^1$ on the level of sentences \cite{vaananen07}, and it is straightforward to see that $\FO(\mathcal D)$ is contained in $\Sigma_1^1$ for all choices of $\mathcal D$. However, for instance, the above-mentioned non-constancy atoms are certainly not definable in Dependence Logic because of Proposition \ref{propo:pres}, since functional dependencies are downwards closed while they are not. 

Classes of dependencies for which safety and definability coincide may be called \emph{closed}: 
\begin{defin}[Closed Classes of Dependencies]
	Let $\mathcal D$ be a class of dependencies. Then $\mathcal D$ is closed if and only if every $\mathcal E$ which is safe for $\mathcal D$ contains only dependencies which are definable in $\FO(\mathcal D)$. 
\end{defin}

A class $\mathcal D$ of dependencies, in other words, is closed if all dependencies that may be added to $\FO(\mathcal D)$ without increasing its expressive power are already expressible in terms of $\FO(\mathcal D)$. The class of all first order dependencies is trivially closed; and, for instance, it follows easily from known results \cite{galliani12} that, since all those dependencies are definable in terms of \emph{independence atoms} \cite{gradel13} $\tuple y \bot_{\tuple x} \tuple z := [R: \tuple x \tuple y \tuple z] \forall \tuple u \tuple v_1 \tuple w_1 \tuple v_2 \tuple w_2 (R \tuple u \tuple v_1 \tuple w_1 \wedge R \tuple u \tuple v_2 \tuple w_2 \rightarrow R \tuple u \tuple v_1 \tuple w_2)$ and \emph{nonemptiness atoms} $\nonempty(x) := [V: x]\exists u V u$, any family containing these two types of dependencies is closed. On the other hand, the family $\mathcal D^\downarrow$ of all downwards closed dependencies is \emph{not} closed in the sense of the above definition, since inclusion atoms and independence atoms are safe for it despite not being downwards closed (and, therefore, not being definable in terms of downwards closed atoms).

The problem of characterizing other, weaker closed classes of dependencies is entirely open, and a complete solution of it would go a long way in providing a classification of the extensions of first order logic via first order dependencies. In particular, the conjecture mentioned in Section \ref{sec:stronglyfo} has the following, equivalent formulation:\\

\noindent \textbf{Open Conjecture (equivalent formulation):} Let $\mathcal D^{\uparrow}$ be the class of all upwards closed dependencies and let $=\!\!(\cdot)$ be the class of all constancy dependencies. Then $\mathcal D^{\uparrow} \cup =\!\!(\cdot)$ is closed. \\

Answering this conjecture, and more in general characterizing the closed families of dependencies, is left to future work. In the rest of this work, a few preliminary results will be presented that provide some information about the properties of the notion of safety.

\section{The Safety of Totality, Inconstancy, Nonemptiness and Possibility}
A natural question to consider to begin exploring the properties of safety is the following: are there dependencies which are safe for \emph{all} families of dependencies $\mathcal D$? As we will see, the answer is positive, as shown by the \emph{totality atoms} $\All(\tuple x) = [R: \tuple x] \forall \tuple v R \tuple v$. 
%

\begin{lemma}
	Let $\phi$ be a $\FO(\All, \mathcal D)$ sentence of the form  $\forall \tuple x_1 \exists \tuple y_1 \ldots \forall \tuple x_n \exists \tuple y_n ( (\theta(\tuple y_n) \hookrightarrow \All(\tuple t)) \wedge \chi(\tuple x, \tuple y))$,  where $\theta$ is first order and $\tuple t$ is a tuple of terms with variables in $\tuple x\tuple y = \tuple x_1 \ldots \tuple x_n \tuple y_1 \ldots \tuple y_n$. \\
	
	Then $\phi$ is logically equivalent to the expression 
	\[
		\forall \tuple z \exists \tuple x_1' \tuple y_1' \ldots \tuple x_n' \tuple y_n'
		\left(\theta(\tuple y_n') \wedge \tuple t' = \tuple z \wedge 
		\forall p q \tuple x_1 \exists \tuple y_1 \ldots \forall \tuple x_n \exists \tuple y_n 
		\left(\bigwedge_i (p=q \wedge \bigwedge_{j\leq i} \tuple x_j = \tuple x'_j) \hookrightarrow \tuple y_i = \tuple y'_i\right) \wedge \chi(\tuple x, \tuple y)\right)
	\]
	where $\tuple z$ is a new tuple of variables of the same arity as $\tuple t$, all $\tuple x'_i$ and $\tuple y'_i$ are tuples of new, pairwise distinct variables of the same arities of the corresponding $\tuple x_i$, $\tuple y_i$, and $\tuple t'$ is obtained from the tuple of terms $\tuple t$ by replacing each variable in $\tuple x_i$ or $\tuple y_i$  with the corresponding variable in $\tuple x'_i$ or $\tuple y'_i$, for all $i$.
	\label{lemma:total}
\end{lemma} 

Using the normal form of Theorem \ref{thm:normalform}, it is now straightforward to show that totality is safe for all families of dependencies:
\begin{theo}[Totality is safe for all $\mathcal D$]
Let $\mathcal D$ be any set of dependencies, and let $\phi \in \FO(\All, \mathcal D)$ be a sentence. Then $\phi$ is equivalent to some $\phi'$ in $\FO(\mathcal D)$. 
	\label{thm:nonemptysafe}
\end{theo}
\begin{proof}
	By Theorem \ref{thm:normalform}, we can assume that $\phi \in \FO(\All, \mathcal D)$ is of the form 
\[
	\forall \tuple  x_1 \exists \tuple y_1 \ldots \forall \tuple x_n \exists \tuple y_n (\bigwedge_k (\theta_k(\tuple y_n) \hookrightarrow \All(\tuple t_k)) \wedge \psi(\tuple x, \tuple y))
\]
	where $\nonempty$ does not occur in $\psi$. Then we get rid of the totality atoms one at a time, using the above lemma and renormalizing. As the normalization procedure of Theorem \ref{thm:normalform} does not introduce further dependency atoms, the procedure will eventually terminate in a sentence without totality atoms. Thus, $\phi$ is equivalent to some sentence $\phi' \in \FO(\mathcal D)$. 
\end{proof}

From the safety of totality it follows at once that all dependencies that are definable in $\FO(\All)$ are also safe. For instance:
\begin{coro}
	The non-constancy dependencies $\nc(R) := \exists \tuple x \exists \tuple y (R \tuple x \wedge R \tuple y \wedge \tuple x \not = \tuple y)$, for which $\M \models_X \nc (\tuple v)$ if and only if $\tuple v$ takes at least two values in $X$, are safe for all $\mathcal D$. So are the nonemptiness dependencies $\nonempty(R) := \exists \tuple v R \tuple v$, for which $\M \models_X \nonempty (\tuple v)$ if and only if $|X(\tuple v)| > 0$.
	\label{coro:ncne}
\end{coro}
\begin{proof}
	Observe that $\nc(\tuple v) \equiv \forall \tuple w (\tuple w \not = \tuple v \hookrightarrow \All(\tuple w))$ and that $\nonempty(\tuple v) \equiv (\tuple v = \tuple v) \wedge \forall w \All(w)$.\footnote{The $\tuple v = \tuple v$ condition is only to make it so that the two expressions have the same free variables. The choice of $\tuple v$ has no other effect on the satisfaction conditions of $\nonempty(\tuple v)$, and one could instead treat $\nonempty := \forall w \All w$ as a ``$0$-ary'' dependency.}	
\end{proof}

Furthermore, additional operators can be shown to be definable in terms of totality (and, hence, not to add to the expressive power of any logic $\FO(\mathcal D)$. For instance, consider the following connective: 
\begin{defin}[Possibility Operator]
	For any family of dependencies $\mathcal D$, let $\FO(\mathcal D, \diamond)$ be the logic obtained by adding to the language of $\FO(\mathcal D)$ a new unary operator $\diamond$ such that, for all models $\M$, teams $X$ and formulas $\psi$ with free variables in $\dom(X)$, 
	\begin{description}
		\item[TS-$\diamond$:] $\M \models_X \diamond \psi$ if and only if there exists some $Y \subseteq X$, $Y \not = \emptyset$, such that $\M \models_Y \psi$. 
	\end{description}
\end{defin}
\begin{coro}
	For all families of dependencies $\mathcal D$, every sentence of $\FO(\mathcal D, \diamond)$ is equivalent to some sentence of $\FO(\mathcal D)$. 
\end{coro}
\begin{proof}
	Observe that $\diamond \psi$ is logically equivalent to $(\nonempty \wedge \psi) \vee \top$. Therefore, every sentence of $\FO(\mathcal D, \diamond)$ is equivalent to some sentence of $\FO(\mathcal D, \nonempty)$; but by Corollary \ref{coro:ncne} this is equivalent to some sentence in $\FO(\mathcal D, \All)$, and by Theorem \ref{thm:nonemptysafe}, every such sentence is equivalent to some sentence of $\FO(\mathcal D)$ as required.
\end{proof}

Results like these ones contribute to the study of Team Semantics not only in the sense that they provide information regarding e.g. the properties of totality, inconstancy and nonemptiness atoms or possibility operators in this context, but also and more importantly because they allow us to use such atoms and operators freely as \emph{tools} for investigating the expressive power of \emph{any other logic} $\FO(\mathcal D)$. For example: 
\begin{coro}
	Let $\subseteq_k$ represent the collection of all $k$-ary inclusion atoms $\tuple x \subseteq \tuple y := [R:\tuple x \tuple y] \forall \tuple u \tuple v (R \tuple u \tuple v \rightarrow \exists \tuple w R \tuple w \tuple u)$ for $|\tuple x| = |\tuple y| = k$, and let $|_k$ represent the $k$-ary exclusion atoms $\tuple x | \tuple y := [R: \tuple x \tuple y] \forall \tuple u \tuple v \tuple u' \tuple v' ((R \tuple u \tuple v  \wedge R \tuple u'\tuple v') \rightarrow (\tuple u \not = \tuple v' \wedge \tuple v \not = \tuple u'))$ (also with $|\tuple x| = |\tuple y| = k$). Then every sentence of $\FO(\subseteq_k)$ is equivalent to some sentence of $\FO(|_k)$. 
\end{coro}
\begin{proof}
	Observe that $\tuple x \subseteq \tuple y$ is logically equivalent to $\exists \tuple z \tuple w (\tuple x | \tuple z \wedge (\tuple w = \tuple y \vee \tuple w = \tuple z) \wedge \All(\tuple w)))$.\footnote{This can be verified by expanding its satisfaction conditions. The intuition behind the above expression is the following: $\tuple w$ must take all possible values, but can take only values which are in $\tuple y$ or are not in $\tuple x$. So if $X$ satisfies the formula then $\overline{X(\tuple x)} \cup X(\tuple y) = M^k$, that is $X(\tuple x) \subseteq X(\tuple y)$.} Thus every sentence of $\FO(\subseteq_k)$ is equivalent to some sentence of $\FO(|_k, \All)$, which -- by the safety of totality -- is equivalent to some sentence of $\FO(|_k)$. 
\end{proof}
This fact could have also been extracted from a careful analysis of known -- and delicate -- equivalences between these logics and fragments of $\Sigma_1^1$.\footnote{More specifically, it is known from \cite{galliani13b} that $\FO(\subseteq_k) \leq \ESO_f(k\text{-ary})$; as it is shown in the arXiv version of \cite{galliani12}, $\FO(|_k)$ is contained in $\FO(=\!\!(\cdots; \cdot)_k)$, where $=\!\!(\cdots; \cdot)_k$ represents $k$-ary \emph{functional dependencies} $=\!\!(\tuple x; y) := [R:\tuple x y] \forall \tuple u v v' (R \tuple u v \wedge R \tuple u v' \rightarrow v =v')$, where $|\tuple x| = k$ -- more specifically, $\tuple x | \tuple y$ is equivalent to $\forall \tuple z \exists p q ( =\!\!(\tuple z; p) \wedge =\!\!(\tuple z; q) \wedge (p=q \hookrightarrow \tuple z \not = \tuple x) \wedge (p \not = q \hookrightarrow \tuple z \not = \tuple y)$; and it is known from \cite{durand2012hierarchies} that $\FO(=\!\!(\cdots; \cdot)_k) = \ESO_f(k\text{-ary})$.} However, the advantage of this approach is that we could obtain our result \emph{directly}, without having to rely on characterizations of these fragments in terms of $\Sigma_1^1$ (which were available for these specific, well-studied logics, but may not be so for other $\FO(\mathcal D)$.).
\section{The Unsafety of Constancy}
As we saw in the previous section, three typical strongly first order dependencies -- that is, totality, nonconstancy and nonemptiness -- are safe for all families of dependencies. A reasonable hypothesis to make at this point would be that the same is true of all strongly first order dependencies. This is not however the case, as constancy atoms are strongly first order but are not safe for all families of dependencies. Indeed, as we will see, graph non-connectedness is definable in terms of constancy and unary inclusion atoms, but not in terms of unary inclusion atoms alone. In keeping with the existing literature on the subject, we will use $=\!\!(x)$ for the atom expressing that $x$ takes a constant value in the team (that is, for $[U: x] \forall v w (U v \wedge U w \rightarrow v = w)$) and $x \subseteq y$ for the atom expressing that all possible values of $x$ are also possible values for $y$ (that is, $[U: x][V:y] \forall v (Uv \rightarrow Vv)$, or equivalently $[R:xy]\forall uv (Ruv \rightarrow \exists w R wu)$). We will use the symbols $=\!\!(\cdot)$ and $\subseteq_1$ for representing these two types of dependencies. Then it is straightforward to see that (as mentioned already in \cite{galliani12}) non-connectedness is definable in $\FO(=\!\!(\cdot), \subseteq_1)$:
\begin{propo}
	The $\FO(=\!\!(\cdot), \subseteq_1)$ sentence $\exists x y (=\!\!(y) \wedge \forall z (E x z \hookrightarrow z \subseteq x) \wedge x \not = y)$ is true in a model $\mathfrak G = (G, E)$ if and only if it is not connected.
	\label{propo:nonconn}
\end{propo}
However, as we will now show, unary inclusion atoms alone do not suffice to define non-connectedness. In particular, for any $n \in \mathbb N$, let the graphs $\mathfrak A_n$ and $\mathfrak B_n$ be constituted respectively by two cycles of length $2^{n+1}$ and by a single cycle of length $2^{n+2}$, as shown in Figure \ref{fig:AnBn}. 
\begin{figure}
\begin{center}
\includegraphics{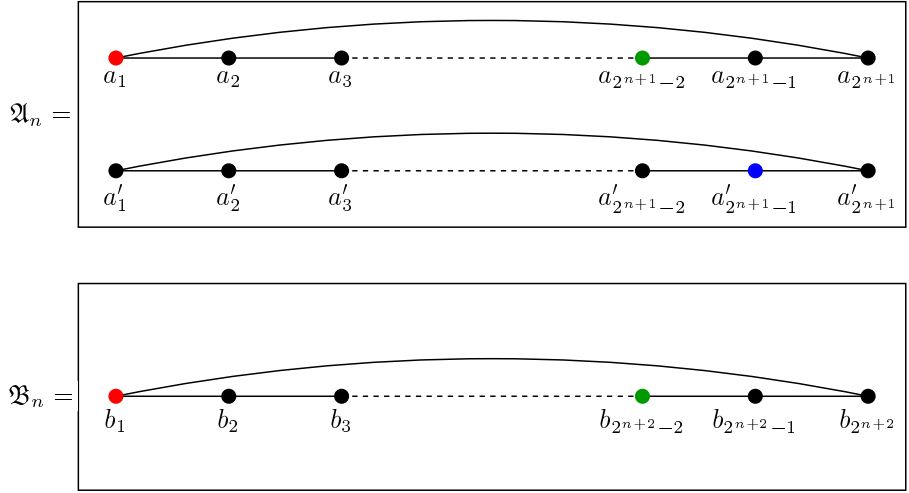}
\end{center}
\caption{The undirected graphs $\mathfrak A_n$ and $\mathfrak B_n$. There is no unary inclusion logic sentences which is true for all $\mathfrak A_n$ and is false for all $\mathfrak B_n$, and therefore non-connectedness is not definable in unary inclusion logic. Note that there exist automorphisms sending any element (red) to any other element of the model, no matter if in the same connected component (green) or in different components (blue).}
\label{fig:AnBn}
\end{figure}

Then, as we will now see, it is not possible to find a $\FO(\subseteq_1)$ sentence that is true in all $\mathfrak A_n$ and false in all $\mathfrak B_n$. This can be proved by means of an Ehrenfeucht-Fra\"iss\'e game defined along the lines of the one for Dependence Logic of \cite{vaananen07}
; but in what follows, a different -- and simpler -- proof will be shown. 

%
%
\begin{lemma}[Automorphisms in $A_n$ and $B_n$]
	Let $\mathfrak G = (G, E)$ be an undirected graph of the form $\mathfrak A_n$ or of the form $\mathfrak B_n$ for some $n \in \mathbb N$, and let $p,q \in G$ be two nodes of this graph. Then there exists an automorphism $\mathfrak f: G \rightarrow G$ of $\mathfrak G$ such that $\mathfrak f(p) = q$. 
\end{lemma}
\begin{defin}[Flattening]
	Let $\phi \in \FO(\subseteq_1)$. Then its \emph{flattening} $\phi^f$ is the first order expression obtained by replacing each inclusion atom $x \subseteq y$ of $\phi$ with the always-true atom $\top$. 
\end{defin}

\begin{lemma}
	For all models $\M$, teams $X$, and formulas $\phi \in \FO(\subseteq_1)$, if $\M \models_X \phi$ then $\mathfrak G \models_X \phi^f$. 
	\label{lemma:flat}
\end{lemma}
\begin{defin}[Team Closure]
Let $X$ be a team over $\mathfrak M$, domain $v_1 \ldots v_n$.  Then $\texttt{Cl}(X)  = \{\mathfrak f(s) : s \in X, \mathfrak f: M \rightarrow M \text{ automorphism}\}$ is the set of all assignments obtained by applying all automorphisms of $\M$ to all assignments of $X$.
\end{defin}
\begin{lemma}
	For all models $\M$ and teams $X$ over $\M$, $\texttt{Cl}(\texttt{Cl}(X)) = \texttt{Cl}(X)$. Furthermore, for all teams $Y$ and $Z$, $\texttt{Cl}(Y \cup Z) = \texttt{Cl}(Y) \cup \texttt{Cl}(Z)$.
	\label{lemma:itercl}
\end{lemma}
\begin{lemma}
	For all models $\M$, all teams $X$ and all first order formulas $\phi$ with free variables in the domain of $X$, 
		$\M \models_X \phi \Leftrightarrow \M \models_{\texttt{Cl}(X)} \phi$.
	\label{lemma:focl}
\end{lemma}

The next lemma is less obvious, and shows that over models such as the $\mathfrak A_n$ and $\mathfrak B_n$ and for teams closed under automorphisms $\FO(\subseteq_1)$ is no more expressive than first order logic:
\begin{lemma}
	Let $\M$ be a model such that for any two points $m_1, m_2 \in M$ there exists an automorphism $\mathfrak f: M \rightarrow M$ of $\M$ such that $\mathfrak f(m_1) = m_2$. 

	Then for all teams $X$ over $\M$ such that $X = \texttt{Cl}(X)$ and all formulas $\phi \in \FO(\subseteq_1)$ with free variables in $\dom(X)$ we have that
		$\M \models_X \phi \Leftrightarrow \M \models_X \phi^f$.
	\label{lemma:clinc}
\end{lemma}
\begin{proof}
	The left to right direction is already taken care of by Lemma \ref{lemma:flat}. The right to left direction is proved via structural induction and presents no particular difficulties. We show in detail the case of inclusion atoms, which is helpful for understanding why $\FO(\subseteq_1)$ is no more expressive than $\FO$ over these types of models. 
	
	As $(v_1 \subseteq v_2)^f = \top$, we need to prove that $\M \models_X v_1 \subseteq v_2$ whenever $X$ is a team whose domain contains the variables $v_1$ and $v_2$ and $X = \texttt{Cl}(X)$. But this is the case. Indeed, suppose that $s(v_1) = m_1$ and $s(v_2) = m_2$. Then by assumption, there is an automorphism $\mathfrak f$ of $\M$ such that $\mathfrak f(m_2) = m_1$, and since $X = \texttt{Cl}(X)$ there exists some assignment $s' \in X$ such that $s'(v) = \mathfrak f(s(v))$ for all $v \in \dom(s)$. This implies in particular that $s'(v_2) = \mathfrak f(s(v_2)) = \mathfrak f(m_2) = m_1 = s(v_1)$; and thus, for any assignment $s \in X$ there exists some assignment $s' \in \texttt{Cl}(X) = X$ such that $s'(v_2) = s(v_1)$. This shows that $\M \models_X v_1 \subseteq v_2$, as required.
\end{proof}

Given the above lemma, the following consequence is immediate: 
\begin{propo}
	Let $\mathfrak G = (G, E)$ be a graph of the form $\mathfrak A_n$ or of the form $\mathfrak B_n$, and let $\phi$ be a $\FO(\subseteq_1)$ sentence over its signature. Then $\mathfrak G \models \phi$ if and only if $\mathfrak G \models \phi^f$.
	\label{propo:noinc}
\end{propo}
\begin{proof}
	By definition, $\mathfrak G \models \phi$ if and only if $\mathfrak G \models_{\{\epsilon\}} \phi$, where $\epsilon$ is the unique empty assignment. But $\{\epsilon\}$ is closed by automorphisms, and therefore $\mathfrak G \models \phi$ if and only if $\mathfrak G \models \phi^f$. 
\end{proof}

However, it can be shown via a standard back-and-forth argument that 
	there is no first order sentence $\phi^f$ that is true in all models of the form $\mathfrak A_n$ and is false in all models of the form $\mathfrak B_n$. 
As a direct consequence of this, of Proposition \ref{propo:nonconn} and of Proposition \ref{propo:noinc} we then have that there exist $\FO(=\!\!(\cdot), \subseteq_1)$ sentences that are not equivalent to any $\FO(\subseteq_1)$ sentence, that is that
\begin{theo}
	Constancy atoms $=\!\!(\cdot)$ are not safe for $\FO(\subseteq_1)$. 
\end{theo}
\section{Conclusions}
In this work, the concept of safe dependencies has been introduced. This notion generalizes the previously considered notion of strongly first order dependencies, and -- aside from being of independent interest -- it is a useful tool for the study of the expressivity (over sentences) of logics based on Team Semantics: indeed, being able to fully characterize the dependencies which are safe for a given logic is the same as fully characterizing the ways in which the language of this logic can be expanded (via dependency atoms) without increasing its overall expressive power. 

A natural point from which to begin the exploration of this notion was to examine the relationship between this notion and the notion of strongly first order dependency itself; and, as we saw, the obvious conjecture according to which a strongly first order dependency must be safe for all families of dependencies does not hold. This shows that the notion of safety is a delicate one -- one that, in particular, is not preserved when additional dependencies are added to the language.\footnote{Recall that a dependency is strongly first order if and only if it is safe for $\FO(\emptyset)$; therefore, in Section 5. we proved that $=\!\!(\cdot)$ is safe for $\FO(\emptyset)$ but not for $\FO(\subseteq_1)$.} The problem of characterizing safe dependencies and closed dependency families is almost entirely open, and steps towards its solution would do much to clarify the properties of logics based on Team Semantics.

We focused exclusively on logics obtained by adding new dependency atoms to the language of First Order Logic (interpreted via Team Semantics). The problems considered here, however, could also be studied as part of a more general theory of \emph{operators} in Team Semantics, for a sufficiently powerful notion of ``operator'' (possibly based on generalized quantifiers and/or on ideas from Transition Semantics \cite{galliani2014transition}). In this wider context, it seems likely that the questions and open conjectures discussed here would be of even harder solution; but on the other hand, it is possible that the study of the expressive power of families of operators (as opposed to dependencies) in Team Semantics would provide useful insights also towards the solution of the questions discussed in this work.\\

\paragraph{Acknowledgments}
The author thanks the anonymous reviewers. Furthermore, he thanks Fausto Barbero for a number of highly useful suggestions and comments. 
\bibliographystyle{eptcs}
\bibliography{biblio}
\end{document}